\documentclass[12pt]{amsart}
\usepackage{hyperref}
\newcommand{\X}{\mathcal{X}}

\newcommand{\C}{\mathbb{C}}

\newcommand{\Q}{\mathcal{Q}}
\newcommand{\R}{\mathbb{R}}
\newcommand{\E}{\mathbb{E}}
\newcommand{\N}{\mathbb{N}}

\newcommand{\Z}{\mathbb{Z}}

\newcommand{\PP}{{\mathbb P}}
\newcommand{\vf}{{\vec q}}

\DeclareMathOperator{\vdc}{-vdC}

\newcommand{\norm}[1]{\left\Vert #1\right\Vert}

\theoremstyle{plain}
\newtheorem{theorem}{Theorem}[section]
\newtheorem{lemma}[theorem]{Lemma}
\newtheorem{proposition}[theorem]{Proposition}

\newtheorem*{theoremA'}{Theorem A'}
\newtheorem*{theoremB'}{Theorem B'}
\newtheorem*{theoremC'}{Theorem C'}

\newtheorem*{theorem*}{Theorem}
\newtheorem*{Correspondence1}{Furstenberg Correspondence Principle}

\newtheorem{corollary}[theorem]{Corollary}

\theoremstyle{definition}

\newtheorem*{example*}{Example}

\theoremstyle{remark}

\begin{document}
\begin{abstract}
If $\vf_1, \ldots, \vf_m\colon\Z\to\Z^\ell$ are polynomials with
zero constant terms and $E\subset\Z^\ell$ has positive upper Banach
density, then we show that the set
$E\cap (E-\vf_1(p-1))\cap\ldots\cap (E-\vf_m(p-1))$
is nonempty for some prime $p$.  We also prove mean convergence
for the associated averages along the prime numbers, conditional
to analogous convergence results along the full integers.  This
generalizes earlier results of the authors, of Wooley and
Ziegler, and of Bergelson, Leibman and Ziegler.
\end{abstract}

\title{The polynomial multidimensional Szemer\'edi Theorem along
shifted primes}
\thanks{The first author was partially supported by Marie Curie IRG 248008,
the second author by the Institut Universitaire de France,
and the third by NSF grant 0900873.}



\author{Nikos Frantzikinakis}
\address{University of Crete,
Department of mathematics, Knossos Avenue, Heraklion 71409, Greece}
\email{frantzikinakis@gmail.com}

\author{Bernard Host}
\address{Laboratoire d'analyse et de
math\'ematiques appliqu\'{e}es, Universit\'e
de Marne la Vall\'ee \& CNRS UMR 8050\\
5 Bd. Descartes, Champs sur Marne\\
77454 Marne la Vall\'ee Cedex 2, France}
\email{bernard.host@univ-mlv.fr}

\author{Bryna Kra}
\address{Department of Mathematics,
Northwestern University \\ 2033 Sheridan Road Evanston \\ IL 60208-2730, USA}
\email{kra@math.northwestern.edu}

\maketitle

\section{Introduction}
\subsection{Background and new results}
Recent advances in ergodic theory and number theory have lead to
numerous results
on patterns in subsets of the integers with positive upper density,
with descriptions of possible restrictions on differences
between successive terms.
In this vein, we show that the parameters in the
polynomial multidimensional Szemer\'edi Theorem
of Bergelson and Leibman~\cite{BL}
can be restricted to the shifted primes.

Let $\PP$ denote the set of prime numbers and
define the {\em
upper Banach density} $d^*(E)$ of a set $E\subset\Z^\ell$
as $d^*(E) = \limsup_{|I|\to\infty}\frac{|E\cap I|}{|I|}$,
where the $\limsup$ is taken over all parallelepipeds $I\subset\Z^\ell$
whose side lengths tend to infinity.

\begin{theorem}
\label{th:main-comb}
Let $\ell,m\in \N$, $\vf_1, \ldots, \vf_m\colon\Z\to\Z^\ell$ be polynomials with $\vf_i(0)=\vec 0$ for
$i=1, \ldots, m$, and let
$E\subset\Z^\ell$ with upper Banach density
$d^*(E) > 0$.  Then the set of integers $n$ such that
$$
 d^*\bigl(E\cap (E-\vf_1(n))\cap\ldots\cap (E-\vf_m(n))
\bigr)>0
$$
has nonempty intersection with $\PP-1$ and $\PP+1$.
\end{theorem}

In fact, our argument shows this intersection has positive
relative density in the shifted primes.

The first result in this direction was due to S\'ark\"ozy~\cite{S},
who used analytic number theory to show that the difference set
$E-E$ for a set   $E$ of positive upper Banach density contains a
shifted prime $p-1$ for some $p\in\PP$ (and similarly, as for all
the results stated  here, a shifted prime of the form $p+1$).
In~\cite{FHK}, relying on strong uniformity results of~\cite{GT}
related to the primes combined with Roth's theorem on arithmetic
progressions,
we took a first step towards a multiple version, showing that such
$E$ contains an arithmetic progression of length $3$ whose common
difference is a shifted prime. This was generalized in two ways.
First, Wooley and Ziegler~\cite{WZ} proved
Theorem~\ref{th:main-comb} for $\ell=1$, relying on a deep ergodic
structure theorem and milder number theoretic input than used
in~\cite{FHK}. More recently, Bergelson, Leibman, and
Ziegler~\cite{BLZ}, proved Theorem~\ref{th:main-comb} for linear
polynomials $\vf_1, \ldots, \vf_m$, by combining the ergodic results
on IP-recurrence of~\cite{FK} and the uniformity results related to
the primes of~\cite{GT}, \cite{GT2}, and~\cite{GTZ} (their proof
also gives the partition version of our main result in full
generality). Theorem~\ref{th:main-comb} generalizes the results
of~\cite{WZ} and~\cite{BLZ}, and
 is in the spirit of~\cite{FHK}, with the main ingredients being
 the number theoretic uniformity
results of~\cite{GT}, \cite{GT2}, and~\cite{GTZ} and a uniform
version of the polynomial Szemer\'edi theorem \cite{BL},
\cite{BHRF}.

By the Furstenberg Correspondence Principle (see Section~\ref{S:FCP}
below), Theorem~\ref{th:main-comb} is equivalent to an ergodic
version and this is the version that we prove.
\begin{theorem}
\label{th:main-ergodic}
Let $\ell\in \N$, $(X, \X, \mu)$ be a probability  space, and let
$T_1, \ldots, T_\ell\colon X\to X$ be  commuting
invertible measure preserving transformations.  Let $m\in \N$, $q_{i,j}\colon\Z\to\Z$
be polynomials with $q_{i,j}(0)=0$ for
$i=1, \ldots, \ell$ and $ j=1,\ldots, m$.  Then for any $A\in\X$ with
$\mu(A)> 0$, the set of integers $n$ such that
$$
 \mu\bigl(A\cap \big(\prod_{i=1}^\ell T_i^{q_{i,1}(n)}\big)A\cap
\ldots \cap \big(\prod_{i=1}^\ell T_i^{q_{i,m}(n)}\big)A \bigr)>0
$$
has nonempty intersection with $\PP-1$ and $\PP+1$.
\end{theorem}
We also prove mean convergence results for the
corresponding multiple ergodic averages over the primes,
conditional on the convergence
of the corresponding averages over the full set of natural numbers
 (in some cases these results are not known).

\begin{theorem}
\label{th:convergence} Let $\ell, m, \in \N$, $(X, \X, \mu)$ be a
probability space,   $T_1, \ldots, T_\ell \colon X\to X$ be
commuting invertible measure preserving transformations, and
$f_1,\ldots, f_m\in L^\infty(\mu)$ be functions.
 For $i=1, \ldots, \ell$ and $ j=1,\ldots, m$,
 let $q_{i,j}\colon\Z\to\Z$
be polynomials. Suppose that the averages
\begin{equation}\label{E:averages1}
\frac{1}{\pi(N)}\sum_{p\in\PP\cap[1,N]} f_1\big((\prod_{i=1}^\ell
T_i^{q_{i,1}(an+b)})x\big) \cdot \ldots \cdot
f_m\big((\prod_{i=1}^\ell T_i^{q_{i,m}(an+b)})x\big),
\end{equation}
converge in $L^2(\mu)$ as $N\to \infty$ for all integers $a,b\geq
1$. Then the averages
\begin{equation}\label{E:averages2}
\frac{1}{\pi(N)}\sum_{p\in\PP\cap[1,N]} f_1 \big((\prod_{i=1}^\ell
T_i^{q_{i,1}(n)})x\big) \cdot \ldots \cdot f_m\big((\prod_{i=1}^\ell
T_i^{q_{i,m}(n)})x\big),
\end{equation}
where $\pi(N)$ denotes the number of primes up to $N$, also converge
in $L^2(\mu)$ as $N\to\infty$.
\end{theorem}

Convergence of~\eqref{E:averages2} when $\ell=m=1$ was proved by
Wierdl~\cite{W} (more generally he showed pointwise convergence, an
issue that we do not address here). When all the transformations are
equal and one restricts to linear polynomials, we proved convergence
of~\eqref{E:averages2}
 in~\cite{FHK}, but for $m\geq 3$ this was conditional upon the results
of~\cite{GT2} and~\cite{GTZ} that were subsequently proven.
In the case where all the transformations are equal,
convergence of~\eqref{E:averages2}
was proved by Wooley and Ziegler in~\cite{WZ}.
Combined with the convergence
results of~\cite{HK3} and~\cite{L}, Theorem~\ref{th:convergence}
recovers the convergence results of~\cite{WZ}.
Using the convergence results of~\cite{Tao}, we obtain the new result
of mean convergence for the linear averages
$$
\frac{1}{\pi(N)}\sum_{p\in\PP\cap[1,N]}
f_1(T_1^{p}x) \cdot \ldots \cdot
f_\ell(T_\ell^{p}x),
$$
and combined with the results of~\cite{CFH}, we
have mean convergence for other new cases, for example the averages
$$
\frac{1}{\pi(N)}\sum_{p\in\PP\cap[1,N]}
f_1(T_1^{p}x) \cdot f_2(T_2^{p^2}x) \cdot
\ldots \cdot
f_\ell(T_\ell^{p^\ell}x).
$$
Combining with the convergence results of ~\cite{A1} and~\cite{A2},
we have mean convergence of the averages
$$
\frac{1}{\pi(N)}\sum_{p\in\PP\cap[1,N]}
f_1(T_1^{p^2}x)\cdot f_2(T_1^{p^2}T_2^px).
$$

\subsection{Strategy of the proof}
 We prove Theorems~\ref{th:main-ergodic}
and~\ref{th:convergence} by reducing the problem to a deep result on
the uniformity of the modified  von Mangoldt function
(Theorem~\ref{T:GTZ} below). The main idea is to compare the
multiple ergodic averages  along the primes with the corresponding
ones along the natural numbers, and show that the difference between
the two converges to zero in mean. Some variation of this idea holds
and is given in Proposition~\ref{P:key}.  The proof of this follows
by successive applications of the van der Corput lemma and a
straightforward PET (polynomial exhaustion technique)
induction argument, reducing the problem to the
aforementioned uniformity result. Given the comparison result of
Proposition~\ref{P:key}, the proof of Theorem~\ref{th:convergence}
follows in a straightforward manner from known convergence results,
and the proof of Theorem~\ref{th:main-ergodic} follows similarly,
with the additional input of a uniform version of the polynomial
Szemer\'edi theorem.

\subsection{Further directions}
Combining the method of this paper with the multiple recurrence
result and methods of~\cite{Lei98}, one can show that
Theorem~\ref{th:main-ergodic} holds under the relaxed assumption
that the transformations $T_1,\ldots, T_\ell$ generate a nilpotent
group (and thus obtain further combinatorial implications, as
in~\cite{Lei98}). Likewise the obvious extension of
Theorem~\ref{th:convergence} to the nilpotent case holds. In both
cases,  the necessary new ingredient is an extension  of the
uniformity estimate of Lemma~\ref{L:key} to the case that the
transformations $T_1,\ldots, T_\ell$ generate a nilpotent group,
which can be proved using the PET induction scheme in~\cite{Lei98}.
We do not carry this out here.

A more challenging problem is the extensions of
Theorems~\ref{th:main-ergodic} and \ref{th:convergence} to sequences
involving fractional powers. For example,
one could hope to show that for any positive real numbers $a$ and $b$,
any $E\subset\Z$ with $d^*(E)>0$ contains patterns of the form
$m, m+[p^a], m+2[p^a]$, or  patterns of the form
$m, m+[p^a], m+[p^b]$ for some $m\in \N$ and $p\in \PP$.
If one is to use the methods of this paper,
one would need to prove an appropriate variant of Lemma~\ref{L:key},
a seemingly nontrivial result.

Lastly, we mention that for two or more transformations, even
the simplest pointwise variants of the mean convergence results we
have established remain open. For example, it is not known if
for a probability space  $(X,\X,\mu)$, measure preserving
transformation $T\colon X\to X$, and functions $f_1,f_2\in
L^\infty(\mu)$,  the averages $\frac{1}{\pi(N)}\sum_{p\in \PP\cap
[1,N]}f_1(T^px)\cdot f_2(T^{2p}x)$, or the averages
$\frac{1}{\pi(N)}\sum_{p\in \PP\cap [1,N]}f_1(T^px)\cdot
f_2(T^{p^2}x)$, converge pointwise as $N\to \infty$.  As a first
step  one could try to prove a pointwise variant of
Theorem~\ref{th:convergence} by using the method of this paper. The
missing ingredient is an appropriate quantitative variant of
Theorem~\ref{T:GTZ}.

\subsection{General conventions and notation}
We denote the positive integers by  $\N=\{1,2,\ldots\}$ and write
$\Z_N=\Z/N\Z$; when needed, the set $\Z_N$ is identified with
$\N\cap [1,N]$.  If $f$ is a measurable function on a measure space
$X$ with transformation $T\colon X\to X$, we write $Tf=f\circ T$. If
$S$ is a finite set and $a\colon S\to \C$, then we write $\E_{n\in
S}a(n)= \frac{1}{|S|}\sum_{n\in S} a(n)$.
We use the symbol $\ll$  when
some expression is  majorized by a constant multiple of some  other expression. If this constant
depends on  variables $k_1,\ldots, k_\ell$,
we write $\ll_{k_1,\ldots, k_\ell}$.  We use $o_N(1)$
to denote a quantity that converges to zero when $N\to \infty$ and
all other parameters are fixed.

\section{Background}
\subsection{Furstenberg correspondence principle}\label{S:FCP}
We state a modification of the  correspondence principle of Furstenberg  (the formulation given is similar to the one in~\cite{BL}):
\begin{Correspondence1}[\cite{Fu1}]
Let $\ell\in \N$ and  $\E\subset \Z^\ell$. There exist a probability
space $(X, \X, \mu)$, commuting  invertible measure preserving
transformations $T_1, \ldots, T_\ell\colon X\to X$,
 and set $A\in\X$ with $\mu(A)=d^*(E)$, such that
$$
   d^*(E \cap (E-\vec n_1)\cap\ldots\cap
   (E-\vec n_\ell))\geq
   \mu\bigl(A\cap (\prod_{i=1}^\ell T_i^{n_{i,1}}A)\cap
\ldots \cap (\prod_{i=1}^\ell T_i^{n_{i,m}}A)
\bigr)
$$
for all $m\in\N$ and
$\vec n_j=(n_{1,j},\ldots,n_{\ell,j}) \in\Z^\ell$ for $j=1,\ldots,m$.
\end{Correspondence1}
In particular, this correspondence shows that  Theorem~\ref{th:main-comb}
follows from Theorem~\ref{th:main-ergodic}.

\subsection{Averages along the primes and weighted averages}

 Let $\Lambda\colon\N\to\R$
denote the von Mangoldt function, taking the value $\log p$ on
a prime $p$ and its powers and $0$ elsewhere, and let
$$
\Lambda'(n) = {\bf 1}_\PP(n)\cdot\Lambda(n)
$$ for $n\in\N$.
Throughout, the roles of $\Lambda$ and $\Lambda'$ are
interchangeable, and all the results can be proven for either
function (as the contribution from prime powers greater than $1$ is
negligible in our averages);  in this article the function
$\Lambda'$ appears more naturally and so we prove the results for
this version.

The following lemma is classical (for a proof, see for example~\cite{FHK}) and  allows us to relate averages over
the primes with weighted averages over the integers:
\begin{lemma}
\label{lemma:compare}
If $a\colon \N\to\C$ is bounded,  then
$$
\Bigl|\frac{1}{\pi(N)}\sum_{p\in\PP, p \leq N}a(p)-\frac{1}{N}\sum_{n=1}^{N} \Lambda'(n)\cdot
a(n)\Bigr| = o_N(1).
$$
\end{lemma}
In particular, the average in~\eqref{E:averages2} is asymptotically
equal to the weighted average over the natural numbers:
$$
\frac{1}{N}\sum_{n=1}^N \Lambda'(n)\cdot f_1 \big((\prod_{i=1}^\ell
T_i^{q_{i,1}(n)})x\big) \cdot \ldots \cdot f_m\big((\prod_{i=1}^\ell
T_i^{q_{i,m}(n)})x\big).
$$

\subsection{Gowers norms}
If $a\colon \Z_N\to \mathbb{C}$, we inductively define:
\begin{equation*}
\norm{a}_{U_1(\Z_N)}=\big|\E_{n\in \Z_N}a(n)\big|
\end{equation*}
and
\begin{equation*}
\norm{a}_{U_{d+1}(\Z_N)}=\Bigl(\E_{h\in \Z_N}\norm{a_h\cdot
\bar{a}}_{U_d(\Z_N)}^{2^d}
\Bigr)^{1/2^{d+1}} ,
\end{equation*}
where $a_h(n) = a(n+h)$. Gowers~\cite{Go} showed that for
$d\geq 2$ this defines a norm on $\Z_N$.

\subsection{Uniformity of the modified von Mangoldt function}
 For $w>2$ let
$$W = \prod_{p\in\PP, p< w} p$$ denote the product of the
primes bounded by $w$. For  $r\in \N$ let
$$\Lambda'_{w,  r}(n)=\frac{\phi(W)}{W}\cdot\Lambda'(Wn+r),
$$
where $\phi$ denotes the Euler function.

The next result is key for our study.  It was obtained in~\cite{GT}
(Theorem 7.2), conditional upon results on the
M\"obius function later obtained in~\cite{GT2}
(Theorem 1.1) and the inverse conjecture for the Gowers norms
(recently proved in~\cite{GTZ}):
\begin{theorem}[Green and Tao (\cite{GT}, \cite{GT2}),
Green, Tao, and Ziegler~\cite{GTZ}]\label{T:GTZ}
With the previous notation, for every $d\in \N$,
 the maximum,  taken over those $r$ between $1$ and $W$  satisfying $(r,W)=1$, of
$$
\norm{(\Lambda'_{w, r}-1)\cdot {\bf 1}_{[1,N]}}_{U_d(\Z_{dN})}
$$
 converges to $0$  as $N\to \infty$ and then $w\to \infty$.
\end{theorem}

 Note that in~\cite{GT} (Theorem~7.2),
the result is stated with $w$ being a specific slowly growing
function of $N$, but the authors also note any sufficiently slowly
growing function of $N$ works too, and this implies our version.
Furthermore, in~\cite{GT} the theorems are stated without the
indicator function ${\bf 1}_{[1,N]}$, but the results of~\cite{GT},
\cite{GT2}, and~\cite{GTZ}, also imply this version.


\section{Comparing averages}

\subsection{PET (polynomial exhaustion technique) induction}
We describe the inductive scheme from~\cite{BL} and
follow the notation and implementation used in~\cite{CFH}.
Let $\ell,m\in\N$. Given $\ell$ ordered families of polynomials
$$
\Q_1=(q_{1,1},\ldots,q_{1,m}) ,\ldots,
\Q_\ell=(q_{\ell,1},\ldots,q_{\ell,m}),
$$
we define an \emph{ordered family $(\Q_1, \ldots, \Q_\ell)$
of $m$ polynomial $\ell$-tuples} by
$$
(\Q_1,\ldots,\Q_\ell)=\big((q_{1,1},\ldots,q_{\ell,1}),
\ldots,(q_{1,m},\ldots,q_{\ell,m})\big).
$$
This gives a concise way of recording
the polynomial iterates that appear in the average of
$$
 f_1(T_1^{q_{1,1}(n)}\cdots T_\ell^{q_{\ell,1}(n)}x) \cdot \ldots
\cdot f_m(T_1^{q_{1,m}(n)}\cdots T_\ell^{q_{\ell,m}(n)}x).
$$
The maximum of the degrees of the polynomials in the families
$\Q_1,\ldots,\Q_\ell$ is called \emph{the degree of  the family} $(\Q_1,\ldots,\Q_\ell)$.

 Fix an integer $s\geq 1$ and consider families of degree $\leq s$.
 For $i=1,\ldots, \ell$, define $\Q_i'$ to be the (possibly empty)
set given by:
$$
\Q_i'=\{ \text{nonconstant } q_{i,j} \in \Q_i\colon q_{i',j}
\text{ is constant for } i'<i\}.
$$

Two polynomials are said to be \emph{equivalent} if they have
the same degree and the same leading coefficient.
For $i=1,\ldots, \ell$ and $j=1,\ldots, s$, we
let $w_{i,j}$ denote the number of distinct non-equivalent
classes of polynomials of degree $j$ in the family  $\Q_i'$.

Define the \emph{(matrix) type} of the family $(\Q_1,\ldots, \Q_\ell)$ to be the matrix
$$
\begin{pmatrix}
w_{1,s}& \ldots &  w_{1,1}\\ w_{2,s}& \ldots& w_{2,1}\\ \vdots & \ldots &\vdots \\
w_{\ell,s}& \ldots& w_{\ell,1}
\end{pmatrix}.
$$

A matrix is said to be of {\em matrix type zero} if all
the $w_{i,j}$ are zero, and this happens exactly when all the polynomials
are constant.

We order the  types lexicographically:
given two $\ell\times s$ matrices $W=(w_{i,j})$
and $W'=(w'_{i,j})$, we say that $W$ is bigger than $W'$,
and write $W>W'$, if $w_{1,d}>w'_{1,d}$, or $w_{1,d}=w'_{1,d}$
and $w_{1,d-1}>w'_{1,d-1}$, $\ldots$,
or $w_{1,i}=w'_{1,i}$ for $i=1,\ldots,d$ and $w_{2,d}>w'_{2,d}$, and so on.
We have:
\begin{lemma}
\label{lem:decreasing2}
Every decreasing sequence of
 types of families of  $\ell$-tuples of polynomials is
 eventually stationary.
\end{lemma}

Thus applying some operation that reduces the type, after
finitely many repetitions, the procedure terminates.
Such an operation is described in the next subsection.

\subsection{The van der Corput operation}
Given a family $\Q=\big(q_1,\ldots,q_m\big)$,   $q\in \Z[t]$, and $h\in\N$, we define
the families $S_h\Q$ and $\Q-q$ as follows:
$$
 S_h\Q=(S_hq_1,\ldots,S_hq_m) \text{ and }
 \Q-q=\big(q_1-q,\ldots,q_m-q\big),
 $$
where $(S_hq)(n)=q(n+h)$.

Given a family of $\ell$-tuples of polynomials  $(\Q_1,\ldots,\Q_\ell)$,
an $\ell$-tuple
$(q_1,\dots,q_\ell)\in (\Q_1,\dots\Q_\ell)$, and $h\in\N$,
define the operation
 $$
(q_1,\ldots, q_\ell,h)\vdc(\Q_1,\ldots, \Q_\ell) =
(\tilde{Q}_{1,h},\ldots \tilde{Q}_{\ell,h}),
 $$
 where
 $$
  \tilde{Q}_{i,h}=(S_h\Q_i-q_i,\Q_i-q_i),
 $$
for $i=1,\ldots,\ell$  (note that this $\tilde{Q}_{i,h}$ is defined
to be the concatenation of two tuples of polynomials).

Starting with a family $(\Q_1,\ldots, \Q_\ell)$, we
successively apply appropriate van der Corput operations
to arrive at constant families of $\ell$-tuples of polynomials.
This is achieved using:
\begin{lemma}[Bergelson and Leibman~\cite{BL}]\label{L:reduceA'}
Let $(\Q_1,\ldots, \Q_\ell)$ be a family of $\ell$-tuples of polynomials with nonzero matrix type.
Then there exists 
$(q_1,\ldots,q_\ell)\in (\Q_1,\ldots,\Q_\ell)$
such that for every $h\in\N$,
the family  $(q_1,\ldots,q_\ell,h)\vdc(\Q_1,\ldots, \Q_\ell)$
has strictly smaller type than $(\Q_1,\ldots, \Q_\ell)$.
\end{lemma}

While this lemma is usually stated to hold for sufficiently
large $h$, this is only in order to maintain extra properties
of the polynomial family (such as being essentially distinct), and
we do not need these properties here.  Thus we are able to phrase this in
the slightly stronger, and easier to use for our purposes, setting
of all $h\in\N$.

Assuming Lemma~\ref{L:reduceA'}, the proof of the next result is standard:
\begin{lemma}\label{L:k(d,m)'}
Let $(\Q_1,\ldots,\Q_\ell)$ be a family of  $m$ polynomial
$\ell$-tuples  with nonzero matrix type.  Suppose that we
successively apply the $(q_1,\ldots,q_\ell,h)\vdc$ operation for
appropriate choices of $q_1,\ldots,q_\ell\in \Z[t]$ and $h\in\N$, as
described in the previous lemma, each time obtaining a family of
$\ell$-tuples of polynomials with strictly smaller type. Then after
a finite number of  operations,  depending only  on $\ell$, $m$, and
the maximum degree of the polynomials (but not on the successive
choices of $h$), we obtain families of $\ell$-tuples of polynomials
of degree $0$.
\end{lemma}

\subsection{Controlling averages}

We state a variation of a classical elementary estimate of van der
Corput.
\begin{lemma} \label{L:VDC2} Let $N\in \N$ and
$v(1),\ldots, v(N)$ be elements of a Hilbert space $\mathcal H$,
with inner product $\langle \cdot, \cdot\rangle$ and norm
$\norm{\cdot}$. Then
$$
\norm{\frac{1}{N}\sum_{n=1}^N v(n)}^2\ll \frac{1}{N^2} \sum_{n=1}^N\norm{v(n)}^2+
 \frac{1}{N}\sum_{h=1}^N
\Bigl\vert{\frac{1}{N}\sum_{n=1}^{N-h}\langle v(n+h), v(n)\rangle}\Bigr\vert.
$$
\end{lemma}
For the case $\mathcal H =\R$ and $\norm{\cdot }=|\cdot |$, the
proof is found, for example in~\cite{KN74}. The proof in the
general case is essentially identical.

Before stating the main lemma
(\ref{L:key})
used to control averages, we
give a simple case that illustrates the technique:
\begin{example*}
Let $a\colon \N\to \C$ be a sequence that satisfies  $a(n)/ n^{1/4}\to 0$.
Let $(X,\X,\mu)$ be a probability space, $T\colon X\to X$
be a measure preserving transformation,
and $f\in L^\infty(\mu)$ be a function bounded by $1$. Then we have that
\begin{equation}
\label{eq:baby} \norm{\frac{1}{N}\sum_{n=1}^N a(n)\cdot
T^{n^2}f}_{L^2(\mu)}\ll \norm{a \cdot {\bf
1}_{[1,N]}}_{U_3(\Z_{3N})} +o_{N}(1).
\end{equation}
To prove this, we apply van der Corput (Lemma~\ref{L:VDC2} for $v(n)=a(n)\cdot  T^{n^2}f$)
and the Cauchy-Schwarz Inequality and we have
\begin{multline*}
\norm{\frac{1}{N}\sum_{n=1}^N a(n)\cdot T^{n^2}f}_{L^2(\mu)}^2\ll\\
\frac{1}{N}\sum_{h_1=1}^N \norm{\frac{1}{N}\sum_{n=1}^{N-h_1} \bar{a}(n+h_1)\cdot a(n)
\cdot T^{2nh_1+h_1^2}f}_{L^2(\mu)}
+ \frac{1}{N^2}\sum_{n=1}^N|a(n)|^2
\end{multline*}
(note that $\norm f_{L^2(\mu)}\leq 1$).  By assumption, the second
term is $o_N(1)$ and we are left with estimating the first term. For
$h_1=1,\ldots,N$, rewriting the interior sum as
$$
\frac{1}{N}\sum_{n=1}^{N} {\bf 1}_{[1,N]}(n+h_1)\cdot \bar{a}(n+h_1)\cdot a(n)
\cdot T^{2nh_1+h_1^2}f,
$$
and applying van der Corput and Cauchy-Schwarz once more,
we have that
\begin{multline*}
\norm{\frac{1}{N}\sum_{n=1}^{N-h_1} \bar{a}(n+h_1)\cdot a(n)
\cdot T^{2nh_1+h_1^2}f}_{L^2(\mu)}^2\ll\\
\frac{1}{N}\sum_{h_2=1}^N \Big|\frac{1}{N}\sum_{n=1}^{ N-h_1-h_2} a(n)\cdot
\bar{a}(n+h_1)\cdot \bar{a}(n+h_2)
\cdot a(n+h_1+h_2)\Big| + \frac{1}{N^2}\sum_{n=1}^N|\bar{a}(n+h_1)
\cdot a(n)|^2.
\end{multline*}
Again, by assumption, the average over $h_1\in \{1,\ldots,N\}$  of the second term is $o_N(1)$.
By further applications of Cauchy-Schwarz, we have that
the eighth power of the $L^2(\mu)$-norm of the
original average is bounded by a constant multiple of
\begin{equation}\label{E:11}
 \frac{1}{N^2}\sum_{1\leq h_1,h_2\leq N} \Big|\frac{1}{N}\sum_{n=1}^{N-h_1-h_2} a(n)
\cdot \bar{a}(n+h_1)\cdot \bar{a}(n+h_2)
\cdot a(n+h_1+h_2)\Big|^2 +o_N(1).
 \end{equation}
On the other hand, letting
$a_N(n)=a(n)\cdot {\bf 1}_{[1,N]}(n)$, for $n=1,\ldots,3N$,
and thinking of $a_N$ as a function $\Z_{3N}\to \C$, we have that
$$
\norm{a_N}_{U_3(\Z_{3N})}^8= \E_{ h_1,h_2 \Z_{3N}} |\E_{n\in
\Z_{3N}} a_N(n) \cdot \bar{a}_N(n+h_1)\cdot \bar{a}_N(n+h_2) \cdot
a_N(n+h_1+h_2)|^2.
$$
(The sums $n+h_1$, $n+h_2$, and $n+h_1+h_2$ are taken modulo
$3N$, and we make
the somewhat less conventional identification of $\Z_{3N}$ with
$[1,\ldots, 3N]$.)
This is greater than or equal to (eliminating values with $N< h_1, h_2
\leq 3N$)
$$
\frac{1}{9N^2}
\sum_{ 1\leq h_1,h_2 \leq N} |\E_{n\in \Z_{3N}} a_N(n)
\cdot \bar{a}_N(n+h_1)\cdot \bar{a}_N(n+h_2)
\cdot a_N(n+h_1+h_2)|^2,
$$
where we maintain the same convention on sums. Since in this
expression we have $1\leq h_1,h_2\leq N$ and  $a_N(n)$ is zero for
$n\in \{N+1,\ldots, 3N\}$, we have that  all  $h_1,h_2,n$ that make
a nonzero contribution to this last average satisfy $1\leq
n+h_1+h_2\leq 3N$. In particular, there are no circular effects and
the last expression is equal to
\begin{multline*}
\frac{1}{9N^2}\sum_{1\leq h_1,h_2 \leq N} \Big|\frac{1}{3N} \sum_{n=1}^{3N}
 a_N(n)
\cdot \bar{a}_N(n+h_1)\cdot \bar{a}_N(n+h_2)
\cdot a_N(n+h_1+h_2)\Big|^2\\
=
\frac{1}{81N^2}\sum_{1\leq h_1,h_2 \leq N} \Big|\frac{1}{N}
\sum_{n=1}^{N-h_1-h_2}
a(n)\cdot \bar{a}(n+h_1)\cdot \bar{a}(n+h_2)
\cdot a(n+h_1+h_2)\Big|^2,
\end{multline*}
 where the sums $n+h_1$, $n+h_2$, and $n+h_1+h_2$ are taken in
$\N$, without reduction modulo $3N$.  But this expression is exactly
 $1/81$ of the average in
\eqref{E:11}. Combining these estimates, we have that the eighth
power of the $L^2(\mu)$-norm of the original averages is bounded by
a constant times $\norm{a_N}_{U_3(\Z_{3N})}^{8}$ plus an $o_N(1)$
term. Thus we have estimate~\eqref{eq:baby}.
 \end{example*}

We now turn to the general case:
\begin{lemma}\label{L:key}
Let $\ell,m\in \N$,  $(X,\X,\mu)$ be a probability space,
$T_1,\dots, T_\ell\colon X\to X$ be commuting invertible measure
preserving transformations, $f_1,\ldots, f_m \in L^\infty(\mu)$ be
functions bounded by $1$, and $q_{i,j}\colon \Z\to \Z$,
$i\in\{1,\ldots,\ell\}$, $j\in\{1,\ldots,m\}$,  be polynomials. Let
$a\colon \N\to \C$  be a sequence of complex numbers satisfying
$a(n)/ n^{c}\to 0$ for every $c>0$. Then there exists $d\in \N$,
depending only on the maximum degree of the polynomials $q_{i,j}$
and the integers $\ell$ and $m$, such that
$$
\norm{\frac{1}{N}\sum_{n=1}^N a(n)\cdot
(\prod_{i=1}^\ell T_i^{q_{i,1}(n)}) f_1
\cdot \ldots \cdot (\prod_{i=1}^\ell T_i^{q_{i,m}(n)})f_m}_{L^2(\mu)}\ll_{d}
\norm{a\cdot {\bf 1}_{[1,N]}}_{U_d(\Z_{dN})} +o_{N}(1).
$$
Furthermore,  the implicit constant is independent of the sequence
$(a(n))_{n\in\N}$, and the $o_{N}(1)$  term depends only the
integer $d$ and on the sequence $(a(n))_{n\in\N}$.
\end{lemma}
\begin{proof}
For $i=1,\ldots,\ell$, let $\Q_i=(q_{i,1},\ldots,q_{i,m})$.
If the matrix type of the family $(\Q_1,\ldots,\Q_\ell)$ is zero,
then  all the polynomials are constant, in which case the conclusion
holds trivially for $d=1$.  If the matrix type is nonzero, then
by Lemma~\ref{L:reduceA'} there exists
$(q_1,\ldots,q_\ell)\in (\Q_1,\ldots,\Q_\ell)$
such that for $h_1\in\N$,
the family  $(q_1,\ldots,q_\ell,h_1)\vdc(\Q_1,\ldots, \Q_\ell)$  has type
strictly smaller than that of $(\Q_1,\ldots, \Q_\ell)$.

As in the model example, using van der Corput and Cauchy-Schwarz, we have
that
\begin{equation}\label{E:first}
\norm{ \frac{1}{N}\sum_{n=1}^N a(n)\cdot (\prod_{i=1}^\ell T_i^{q_{i,1}(n)})
f_1\cdot \ldots \cdot (\prod_{i=1}^\ell T_i^{q_{i,m}(n)})
f_m}_{L^2(\mu)}^{2^{d+1}}
\end{equation}
is bounded by an $o_N(1)$ term plus  a constant multiple of
$$
\frac{1}{N}\sum_{h_1=1}^N
  \norm{\frac{1}{N}\sum_{n=1}^{N-h_1} \bar{a}(n+h_1) \cdot
 a(n)\cdot
(\prod_{i=1}^\ell T_i^{q_{h_1,i,1}(n)})g_1 \cdot \ldots \cdot
(\prod_{i=1}^\ell T_i^{q_{h_1,i,2m}(n)})g_{2m}}_{L^2(\mu)}^{2^d},
$$
where $(q_{h_1,1,j},\ldots q_{h_1,\ell,j})\in
(q_1,\ldots,q_\ell,h_1)\vdc(\Q_1,\ldots, \Q_\ell)$ for every $h_1\in
\N$ and  $j=1,\ldots, 2m$  and each function $g_k$ is equal to one
of the functions $f_j$.
 If the new family of polynomials has zero matrix type, we stop.
If not, as in the model example, we continue to use
van der Corput and Cauchy-Schwarz to bound the average over $n$.
By Lemma~\ref{L:k(d,m)'}, after a finite number of steps, depending
only on the maximum degree of the polynomials $q_{i,j}$ and the
integers $\ell$ and $m$, we have families of polynomials with
zero matrix type. Assume that this takes $d$ steps.
We deduce that the expression~\eqref{E:first}
is bounded by a $o_N(1)$ term (using the
assumption that $a(n)/n^c\to 0$ for every $c>0$ to control
the lower order terms) plus  a constant multiple of
 $$
\frac{1}{N^d}\sum_{ 1\leq h_1,\ldots, h_d \leq N} \Big|\frac{1}{N} \sum_{n=1}^{N-h_1-\cdots-h_d}
 a(n)
\cdot \bar{a}(n+h_1)\cdot  \bar{a}(n+h_2)\cdot\ldots\cdot  a(n+h_1+\cdots+h_d)
\Big|^2.
$$
(Note that the last occurrence of $a$ in this
expression may actually be $\bar{a}$, depending on the parity of
$d$.)
As in the model example, we see that this last average is bounded by a constant
(equal to $d^d$) times
$$
 \norm{
a\cdot {\bf 1}_{[1,N]}}_{U_{d+1}(\Z_{(d+1)N})}^{2^{d+1}},
$$
completing the proof.
\end{proof}

\subsection{Comparing averages}
The key result needed to compare averages over the primes and
over the integers is  (recall that
$W = \prod_{p\in\PP, p< w} p$ denotes the product of the
primes bounded by $w$):

\begin{proposition}\label{P:key}
Let $\ell,m\in \N$,  $(X,\X,\mu)$ be a probability space, $T_1,\dots, T_\ell\colon X\to X$ be commuting invertible measure preserving transformations, $f_1,\ldots, f_m \in L^\infty(\mu)$ be functions, and
$q_{i,j}\colon \Z\to \Z$, $i\in\{1,\ldots,\ell\}$, $j\in\{1,\ldots,m\}$,  be polynomials.
Then the maximum,   taken over those $r$ between $1$ and $W$  satisfying $(r,W)=1$, of the $L^2(\mu)$-norm of
$$
\frac{1}{N}\sum_{n=1}^N(\Lambda'_{w,r}(n)-1) \cdot
(\prod_{i=1}^\ell T_i^{q_{i,1}(Wn+r)})
f_1
\cdot \ldots \cdot (\prod_{i=1}^\ell T_i^{q_{i,m}(Wn+r)})f_m
$$
converges to $0$
as $N\to \infty$ and then $w\to \infty$.
\end{proposition}
\begin{proof}
We can assume that all functions are bounded by $1$.
We apply Lemma~\ref{L:key}
for $a_{w,r}(n)=\Lambda'_{w,r}(n)-1 $ for $w,r\in\N$,
and the family of polynomials $q_{i,j}(Wn+r)$. Let
$\Z_W^*=\{r\in [1,W]\colon (r,W)=1\}$.  We get that
there  exists $d\in \N$,  independent of $w$ and $r$, such that
\begin{multline*}
\max_{r\in \Z_W^*} \norm{\frac{1}{N}\sum_{n=1}^N (\Lambda'_{w,r}(n)-1)\cdot
(\prod_{i=1}^\ell T_i^{q_{i,1}(Wn+r)}) f_1
\cdot \ldots \cdot (\prod_{i=1}^\ell
T_i^{q_{i,m}(Wn+r)})f_m}_{L^2(\mu)}\ll_{d}\\
\max_{r\in \Z_W^*}\norm{(\Lambda'_{w,r}-1)\cdot {\bf 1}_{[1,N]}}_{U_d(\Z_{dN})} +o_N(1)
\end{multline*}
where the term $o_N(1)$ depends only on the integers $d$ and  $w$.
The result now follows from Theorem~\ref{T:GTZ}.
\end{proof}

\section{Proof of the main results}
\subsection{Proof of Theorem~\ref{th:main-ergodic}}
We use the following uniform multiple recurrence result,
proved in the same way as Theorem 3.2 is proved in~\cite{BHRF}:
\begin{theorem}
\label{T:UnifPolSz1}
Let $(X, \X, \mu)$ be a probability  space and  $T_1, \ldots, T_\ell\colon X\to X$ be  commuting
invertible measure preserving transformations.  Let $q_{i,j}\colon\Z\to\Z$
be polynomials with $q_{i,j}(0)=0$ for
$i=1, \ldots, \ell$ and $ j=1,\ldots, m$.
Then for any $A\in\X$ with
$\mu(A)> 0$, there exists a positive constant $c$,
depending only on   $\mu(A)$ and the polynomials $q_{i,j}$, such that
$$
\liminf_{N\to \infty} \frac{1}{N}\sum_{n=1}^N
\mu\bigl(A\cap (\prod_{i=1}^\ell T_i^{q_{i,1}(n)}A)\cap
\ldots \cap (\prod_{i=1}^\ell T_i^{q_{i,m}(n)}A)
\bigr)\geq c.
$$
\end{theorem}
It is important to note that
the constant $c$ does not depend on the transformations $T_1,\ldots, T_\ell$. This observation
enables us to prove a uniform multiple recurrence result
more suitable for our purposes (the uniformity in $W$ is crucial):
\begin{corollary}
\label{C:UnifPolSz2}
Let $(X, \X, \mu)$ be a probability  space and  $T_1, \ldots, T_\ell\colon X\to X$ be  commuting
invertible measure preserving transformations.  Let $q_{i,j}\colon\Z\to\Z$
be polynomials with $q_{i,j}(0)=0$ for
$i=1, \ldots, \ell$ and $ j=1,\ldots, m$.
Then for any $A\in\X$ with
$\mu(A)> 0$, there exists a positive constant $c$,
depending on $\mu(A)$ and the polynomials $q_{i,j}$, such that
for every $W\in \N$, we have
$$
\liminf_{N\to \infty} \frac{1}{N}\sum_{n=1}^N
\mu\bigl(A\cap (\prod_{i=1}^\ell T_i^{q_{i,1}(Wn)}A)\cap
\ldots \cap (\prod_{i=1}^\ell T_i^{q_{i,m}(Wn)}A)
\bigr)\geq c.
$$
\end{corollary}
\begin{proof}
We write the proof for $\ell,m=1$, as the general case follows in an
analogous manner.  Let $(X,\X,\mu)$ be a probability  space and let
$T\colon X\to X$ be an invertible measure preserving transformation.
Let $q(n)=c_{1}n+\cdots+c_{d}n^d$, where $c_1,\ldots,c_d\in \Z$ and
$d\in \N$. Given $A\in \X$ and  $W\in \N$, we have that
$$
\mu(A\cap T^{q(Wn)}A)=
\mu\bigl(A\cap (\prod_{i=1}^{d} S_i^{n^i}A)\bigr)
$$
where $S_{i}=T^{c_iW^i}$ for $i=1,\ldots,d$.
The result now follows from Theorem~\ref{T:UnifPolSz1}.
\end{proof}

Combining  Proposition~\ref{P:key} and Corollary~\ref{C:UnifPolSz2},
 we have that for sufficiently large $w\in \N$,
$$
\liminf_{N\to\infty}\frac{1}{N}\sum_{n=1}^N\Lambda'_{w,1}(n)\cdot
\mu\bigl(A\cap (\prod_{i=1}^\ell T_i^{q_{i,1}(Wn)}A)\cap
\ldots \cap (\prod_{i=1}^\ell T_i^{q_{i,m}(Wn)}A)
\bigr)> 0.
$$
By Lemma~\ref{lemma:compare}, the conclusion of Theorem~\ref{th:main-ergodic}
is satisfied for a set of $n$ with positive relative
density in the shifted primes $\PP-1$.

A similar argument holds for the shifted primes $\PP+1$.

\subsection{Proof of Theorem~\ref{th:convergence}}
To complete the proof, we follow the method used in~\cite{FHK}.
By Lemma~\ref{lemma:compare},
it suffices to prove convergence in $L^2(\mu)$ for the corresponding weighted averages

$$
A(N) = \frac{1}{N}\sum_{n=1}^{N}\Lambda'(n)\cdot (\prod_{i=1}^\ell
T_i^{q_{i,1}(n)})f_1 \cdot \ldots \cdot (\prod_{i=1}^\ell
T_i^{q_{i,m}(n)})f_m.
$$
Equivalently, it suffices to show that the sequence of functions $(A(N))_{N\in \N}$
is Cauchy in $L^2(\mu)$.

Let $\varepsilon>0$. Fix $w,r\in \N$, and let
$$
B_{w,r}(N) = \frac{1}{N}\sum_{n=1}^{N} (\prod_{i=1}^\ell
T_i^{q_{i,1}(Wn+r)})f_1 \cdot \ldots \cdot (\prod_{i=1}^\ell
T_i^{q_{i,m}(Wn+r)})f_m.
$$
(As before, $W$ denotes the product of primes bounded by $w$.)
By Proposition~\ref{P:key},
we have that for some $w_{0}\in \N$ (and corresponding $W_0\in\N$), if $N$ is large enough, then
\begin{equation}\label{E:a1a}
\norm{A(W_{0}N)-\frac{1}{\phi(W_0)}\sum_{1\leq r \leq W_0,
(r,W_0) = 1}B_{w_{0},r}(N)}_{L^{2}(\mu)}\leq \varepsilon/6.
\end{equation}
By assumption, for $r=1,\ldots, W_0$,  the sequence $(B_{w_{0},r}(N))_{N\in\N}$ converges in $L^2(\mu)$. Therefore,
   if $M$ and $N$ are
sufficiently large, then for $r=1,\ldots, W_0$ we have
\begin{equation}\label{E:a1b}
\norm{B_{w_{0},r}(N)  - B_{w_{0},r}(M)}_{L^{2}(\mu)}\leq \varepsilon/6.
\end{equation}
  Combining  \eqref{E:a1a} and \eqref{E:a1b} we have that if $M$ and $N$
are sufficiently large, then
\begin{equation}\label{E:a2}
\norm{A(W_{0}N) - A(W_{0}M)}_{L^{2}(\mu)}\leq \varepsilon/2.
\end{equation}
Lastly, for $r=1,\ldots, W_0$,  we have
\begin{equation}\label{E:a3}
\lim_{N\to\infty}\norm{A(W_0N+r)-A(W_0N)}_{L^2(\mu)}=0.
\end{equation}
Combining  \eqref{E:a2} and  \eqref{E:a3}, it follows that if $M$ and $N$ are sufficiently large, then
$$
\norm{A(N)-A(M)}_{L^2(\mu)}\leq \varepsilon.
$$
Therefore, the sequence $(A(N))_{N\in\N}$ is Cauchy in $L^2(\mu)$,  completing the proof of Theorem~\ref{th:convergence}.

\end{document}